\documentclass[12pt,letterpaper]{article}
\usepackage[utf8]{inputenc}
\usepackage[T1]{fontenc}
\usepackage{amsmath}
\usepackage{amsfonts}
\usepackage{amssymb, enumerate}
\usepackage{amsthm}
\usepackage{graphicx}
\usepackage[margin=1.2in]{geometry}
\usepackage[colorlinks,linkcolor=blue,citecolor=red]{hyperref}
\usepackage{xcolor}
\usepackage{tikz}
\tikzset{every node/.style={circle,fill, inner sep=0pt, minimum size=1.5mm}}

\newcounter{tbox}
\newcommand{\tbox}[1]{\vspace*{0.3cm}\refstepcounter{tbox}\noindent{ \parbox{\textwidth}{(\thetbox) \emph{#1}}}\vspace*{0.3cm}}

\newtheorem{theorem}{Theorem}[section]
\newtheorem{lemma}[theorem]{Lemma}

\title{The $r$-coloring and maximum stable set problem in hypergraphs with bounded matching number and edge size}
\author{Yanjia Li\thanks{Current address: Georgia Institute of Technology, Atlanta, GA.} \, and Sophie Spirkl\thanks{We acknowledge the support of the Natural Sciences and Engineering Research Council of Canada (NSERC), [funding
reference number RGPIN-2020-03912]. \newline Cette recherche a \'et\'e financ\'ee par le Conseil de recherches en sciences naturelles et en g\'enie du Canada (CRSNG),
[num\'ero de r\'ef\'erence RGPIN-2020-03912].}\\University of Waterloo, Waterloo, ON N2L 3G1, Canada}

\begin{document}
\maketitle

\begin{abstract}
Motivated by the analogous questions in graphs, we study the complexity of coloring and stable set problems in hypergraphs with forbidden substructures and bounded edge size. Letting $\nu(G)$ denote the maximum size of a matching in $H$, we obtain complete dichotomies for the complexity of the following problems parametrized by fixed $r, k, s \in \mathbb{N}$: 
\begin{itemize}
    \item \textsc{$r$-Coloring} in hypergraphs $G$ with edge size at most $k$ and $\nu(G) \leq s$;
    \item \textsc{$r$-Precoloring Extension} in $k$-uniform hypergraphs $G$ with $\nu(G) \leq s$;
    \item \textsc{$r$-Precoloring Extension} in hypergraphs $G$ with edge size at most $k$ and $\nu(G) \leq s$;
    \item \textsc{Maximum Stable Set} in $k$-uniform hypergraphs $G$ with $\nu(G) \leq s$;
    \item \textsc{Maximum Weight Stable Set} in $k$-uniform hypergraphs with $\nu(G) \leq s$;
\end{itemize}
as well as partial results for \textsc{$r$-Coloring} in $k$-uniform hypergraphs $\nu(G) \leq s$. We then turn our attention to \textsc{$2$-Coloring} in 3-uniform hypergraphs with forbidden induced subhypergraphs, and give a polynomial-time algorithm when restricting the input to hypergraphs excluding a fixed one-edge hypergraph. Finally, we consider linear 3-uniform hypergraphs (in which every two edges share at most one vertex), and show that excluding an induced matching in $G$ implies that $\nu(G)$ is bounded by a constant; and that $3$-coloring linear $3$-uniform hypergraphs $G$ with $\nu(G) \leq 532$ is NP-hard.
\end{abstract}

\section{Preliminaries}
	A \emph{hypergraph} $G$ is a pair $(V,E)$ where $V$ is a finite set, and $E\subseteq 2^{V}\setminus \{\emptyset\}$. $V$ is called the set of \emph{vertices} and $E$ is called the set of \emph{edges}. For a hypergraph $G=(V,E)$, we define $V(G)=V$ and $E(G)=E$. For $k\in \mathbb{N}$, we say that $G$ is \emph{$k$-uniform} if $|e|=k$ for all edges $e\in E$, and $G$ is \emph{$k$-bounded} if $|e|\leq k$ for all edges $e\in E$. A 2-uniform hypergraph is simply called a graph.
	
	Given a set $X\subseteq V(G)$, $G[X]=(X,\{e\in E(G):e\subseteq X\})$ is an \emph{induced hypersubgraph} of $G$. A \emph{matching} of $G$ is a set of pairwise disjoint edges. A \emph{maximal matching} of $G$ is a matching which is maximal with respect to inclusion. For a hypergraph $G$, we denote $\nu (G)$ be the maximum integer $s$ such that $G$ contains a matching of size $s$. A set $S\subseteq V(G)$ of $G$ is \emph{stable} if $e\cap S\neq e$ for every $e\in E(G)$. 
	
	We use $[r]$ to denote the set $\{1, \dots, r\}$. Given a hypergraph $G$ and a positive integer $r$, a function $c:V(G)\rightarrow [r]$ is an \emph{$r$-coloring} of $G$ if for all $i\in [r]$, $c^{-1}(i)$ is a stable set in $G$. $G$ is \emph{$r$-colorable} if there exists an $r$-coloring of $G$. The \emph{chromatic number} of $G$, denoted $\chi(G)$, is the minimum integer $r$ such that $G$ is $r$-colorable.
	
	A function $c:X\rightarrow [r]$ for some $X\subseteq V(G)$ is a \emph{partial $r$-coloring} of $G$ if $c$ is an $r$-coloring of $G[X]$. For convenience, we also denote a partial coloring as $(X,c)$. Given a partial $r$-coloring $(X,c)$ of $G$, an \emph{$r$-precoloring extension} of $(X,c)$ is a partial $r$-coloring $(X',c')$ with $c'(v)=c(v)$ for all $v\in X$, and $X\subset X'$. We say that a partial coloring $(X,c)$ \emph{$r$-extends} to $G$ if there is an $r$-precoloring extension $(V(G),c')$ of $(X,c)$.
	
	For a fixed integer $r$, the \textsc{Hypergraph $r$-Coloring Problem} is to decide whether a given hypergraph $G$ is $r$-colorable, and the \textsc{Hypergraph $r$-Precoloring Extension Problem} is to decide given a hypergraph $G$ and a partial $r$-coloring $(X,c)$, whether $(X, c)$ $r$-extends to $G$.
	
	The \textsc{Graph $r$-Coloring Problem} is a well-known NP-hard problem: 
	
	\begin{theorem} [Karp \cite{Karp}] \label{rColNP}
		For every fixed integer $r$ with $r\geq 3$, the \textsc{Graph $r$-Coloring Problem} is NP-complete.
	\end{theorem}

	However, placing structural restrictions on the input graphs may make the problem easier. This is well-studied for graphs; see \cite{Col-rP2} for a survey of complexity results for coloring problems in graphs with forbidden induced subgraphs. Few graphs $H$ have the property that $r$-coloring $H$-free graphs can be solved in polynomial time for all $r$. The following result shows that graphs of the form $H = sP_2$ (that is, matchings) have this property: 
	\begin{theorem} [Dabrowski, Lozin, Raman, Ries \cite{dabrowski2012colouring} based on results of Balas and Yu \cite{MaxStb-IndMatching} and Tsukiyama, Ide, Ariyoshi, Shirakawa \cite{MaxStb-EnumMatching}; explicitly stated in Golovach, Johnson, Paulusma and Song \cite{Col-rP2}] \label{rP2}
		For fixed positive integers $r$ and $s$, the \textsc{Graph $r$-Coloring Problem} restricted to $sP_2$-free graphs is polynomial-time solvable.
	\end{theorem}

	The hypergraph coloring problem is a natural extension of the graph coloring problem; see the survey \cite{HpgColSurvey}. The following result shows that the problem is NP-hard, even in uniform hypergraphs. 
	\begin{theorem} [
	Phelps and R\"odl \cite{Linear&gadget}] \label{3Unf2ColNP}
        For all $k \geq 3$ and $r \geq 2$, the \textsc{$k$-Uniform Hypergraph $r$-Coloring Problem} is NP-complete.
	\end{theorem}
	
    It is natural to ask which restrictions of the input hypergraphs make the \textsc{Hypergraph $r$-Coloring Problem} polynomial-time solvable.	In this paper, we mainly focus on bounded or uniform hypergraphs. In view of Theorem \ref{rP2}, we are particularly interested in excluding a matching; but as it turns out, even bounding the maximum size of a matching (a much stronger condition than excluding a large matching as an induced sub(hyper)graph, as is the case in Theorem \ref{rP2}) does not always lead to a polynomial-time algorithm. We prove the following dichotomies:
	\begin{theorem}\label{NP-Dic}
		Let $k,r$ and $s$ be positive integers with $k,r\geq 2$. The \textsc{$k$-Bounded Hypergraph $r$-Coloring Problem}, the \textsc{$k$-Bounded Hypergraph $r$-Precoloring Extension Problem} as well as the \textsc{$k$-Uniform Hypergraph $r$-Precoloring Extension Problem}, restricted to hypergraphs with $\nu(G)\leq s$, are polynomial-time solvable if 
		\begin{itemize}
			\item $s\leq r-1$, or
			\item $k= 3$ and $r=2$, or
			\item $k=2$,
		\end{itemize}
		and NP-complete otherwise.
	\end{theorem}
	
	We also show the following result:
	\begin{theorem} \label{NP-Gap}
			Let $k,r$ and $s$ be positive integers with $k,r\geq 2$. The \textsc{$k$-Uniform Hypergraph $r$-Coloring Problem} restricted to hypergraphs with $\nu(G)\leq s$ is polynomial-time solvable if
		\begin{itemize}
			\item $s\leq r-1$, or
			\item $k=3$ and $r=2$, or
			\item $k=2$,
		\end{itemize}
		and is NP-complete if
		\begin{itemize}
			\item $s\geq (r-1)k+1$, and
			\item $k\geq 4$ or $r\geq 3$.
		\end{itemize}  
	\end{theorem}
	
	Theorem \ref{rP2} is based on a result of \cite{MaxStb-IndMatching} that $sP_2$-free graphs have only polynomially many maximal (with respect to inclusion) stable sets. Using this, \cite{MaxStb-IndMatching} gave a polynomial-time algorithm for finding a maximum (weight) stable set in an $sP_2$-free graph. We ask an analogous question in hypergraphs with bounded maximum matching size, and prove: 
	\begin{theorem} \label{stable-set}
		For fixed positive integers $k$ and $s$ with $k\geq 3$, the \textsc{$k$-Uniform Hypergraph Maximum Stable Set Problem} restricted to hypergraphs with $\nu(G)\leq s$ is polynomial-time solvable, and the \textsc{$k$-Uniform Hypergraph Maximum Weight Stable Set Problem} restricted to hypergraphs with $\nu(G)\leq s$ is NP-complete.
	\end{theorem}
	
	We also give a first result for excluding an induced subhypergraph: 
	\begin{theorem} \label{subhyper}
	Let $t \in \mathbb{N}$ be fixed, and let $H$ be the 3-uniform hypergraph with $t+3$ vertices and one edge. Then there is a polynomial-time algorithm for the \textsc{$3$-Bounded Hypergraph $2$-Coloring Problem} restricted to $H$-free hypergraphs.  
	\end{theorem}
	
	Finally, we consider linear hypergraphs. A hypergraph $G$ is \emph{linear} if $|e \cap e'| \leq 1$ for every two distinct $e, e' \in E(G)$. The restriction to linear hypergraphs does not affect NP-hardness: 
	
	\begin{theorem}[Phelps and Rödl \cite{Linear&gadget}]
	For every $r \geq 2$, the \textsc{3-Uniform Hypergraph $r$-Coloring Problem} restricted to linear hypergraphs is NP-complete. 
	\end{theorem}

    The following result gives an algorithm for $2$-coloring certain linear hypergraphs:
	\begin{theorem} [Chattopadhyay and Reed \cite{Linear-BddDegree}]
	    There is a polynomial-time algorithm for the \textsc{$k$-Uniform Hypergraph 2-Coloring Problem} restricted to linear hypergraphs with maximum degree bounded by a function of $k$.
	\end{theorem}
	
	We ask how our results extend to linear 3-uniform hypergraphs. For $s \in \mathbb{N}$, we let $M_s$ denote the 3-uniform hypergraph with $3s$ vertices and $s$ pairwise disjoint edges. We show that in linear hypergraphs, excluding a fixed induced matching implies bounded matching number, which immediately implies (assuming Theorems \ref{NP-Dic} and \ref{stable-set}): 
	\begin{theorem}
	    Let $s \in \mathbb{N}$. The \textsc{3-Uniform Hypergraph 2-Coloring Problem}, the \textsc{3-Uniform Hypergraph 2-Precoloring Extension Problem}, and the \textsc{3-Uniform Hypergraph Maximum Stable Set Problem} restricted to linear $M_s$-free hypergraphs are polynomial-time solvable. 
	\end{theorem}
	
	We prove:
	\begin{theorem} \label{thm:linear}
	    The \textsc{3-Uniform Hypergraph 3-Coloring Problem} restricted to linear hypergraphs $G$ with $\nu(G) \leq s$ is NP-complete for all $s \geq 532$.
	\end{theorem}
	
	We will give polynomial-time algorithms for the case $k=3$ and $r=2$ and the case $s\leq r-1$ of Theorems \ref{NP-Dic} and \ref{NP-Gap} in Section \ref{Sec-3Bdd2Col} and Section \ref{Sec-kBddrCol-sBdd} respectively. In Section \ref{Sec-NP}, we will talk about some NP-hard cases and complete the proof of Theorems \ref{NP-Dic} and \ref{NP-Gap}. In Section \ref{Sec-StbSet}, we will prove Theorem \ref{stable-set}. In Section \ref{Sec-OneEdge}, we will prove Theorem \ref{subhyper}. Finally, in Section \ref{sec:linear}, we will prove results about linear hypergraphs, including Theorem \ref{thm:linear}.

\section{Algorithm for the case $k=3$ and $r=2$}

\label{Sec-3Bdd2Col}

	In this section, we prove:
	\begin{theorem}\label{3Bdd2Col}
		For every fixed positive integer $s$, the \textsc{ 3-Bounded Hypergraph 2-Coloring Problem} restricted to hypergraphs with $\nu(G)\leq s$ is polynomial-time solvable.
	\end{theorem}

A common strategy for coloring algorithms is using an algorithm for \textsc{2-SAT} as a subroutine: Given an instance $I$ consisting of $n$ Boolean variables and $m$ clauses, each of which contains 2 literals, the \textsc{2-Satisfiability Problem (2-SAT)} is to decide whether there exists a truth assignment for every variable such that every clause contains at least one true literal. We say $I$ is \textit{satisfiable} if it admits such an assignment.	
	\begin{theorem}[Krom \cite{2SAT}; Aspvall, Plass and Tarjan \cite{2SAT-time}]
		The \textsc{2-SAT Problem} can be solved in time $O(n+m)$, where $n$ is the number of variables and $m$ is the number of clauses.
	\end{theorem}
	
Given two partial $r$-coloring collections $\mathcal{C}$, $\mathcal{C}'$ of a hypergraph $G$, we say $\mathcal{C}$ and $\mathcal{C}'$ are \emph{$r$-equivalent} if $\mathcal{C}$ contains a partial $r$-coloring $c_1$ which $r$-extends to $G$ if and only if $\mathcal{C'}$ contains a partial $r$-coloring $c_2$ which $r$-extends to $G$. We say $(X,c)$ is $r$-equivalent to $\mathcal{C}$ if the collection $\{(X,c)\}$ is $r$-equivalent to $\mathcal{C}$. We say that $\mathcal{C}$ is \emph{$r$-equivalent} to $G$ if $G$ is $r$-colorable if and only if $\mathcal{C}$ contains a partial $r$-coloring which $r$-extends to $G$.

	\begin{proof}[Proof of Theorem \ref{3Bdd2Col}]
		Let $G$ be a $3$-bounded hypergraph with $\nu(G)\leq s$. First, we create a collection $\mathcal{C}$ of partial 2-colorings as follows. We fix a maximal matching $F$ of $G$. We define the set $X^F=\cup_{e\in F} e$. Let $\mathcal{C}$ be the set of all partial 2-colorings $(X^F, c:X^F\rightarrow [2])$ of $G$.
		
		We claim that the collection $\mathcal{C}$ has the following three properties. The theorem follows immediately from these properties.
		
		\tbox{ $\mathcal{C}$ is 2-equivalent to $G$.\label{3U2Col-equiv}} 
		
			It suffices to show that if $G$ has a 2-coloring $c$, then there is a partial 2-coloring in $\mathcal{C}$ which has a 2-precoloring extension. Let $c$ be a 2-coloring of $G$. Consider the partial 2-coloring $(X^F, c|_{X^F})\in \mathcal{C}$. Then $c$ is a 2-precoloring extension of $c|_{X^F}$. Thus we have proved \eqref{3U2Col-equiv}.
		
		\tbox{ $\mathcal{C}$ can be computed in time $O(n^{3})$.}
			
			Let $|V(G)|=n$. Since $G$ is 3-bounded, $|E(G)|\leq O(n^3)$. We can go through all edges and construct a maximal matching $F$ in time $O(n^3)$. Checking whether $(X,c)$ is a partial 2-coloring takes time $O(1)$, as the size of $X$ is bounded. Since $|F|\leq \nu (G)\leq s$, we have $|\mathcal{C}|\leq 2^{3s}=O(1)$. Thus, $\mathcal{C}$ can be constructed from $F$ in time $O(n^{3})$.
					
		\tbox{For every partial 2-coloring $c'$ in $\mathcal{C}$,  whether $c'$ has a 2-precoloring extension $(V(G),c)$ can be decided in polynomial time.}
		
			Let $(X^F,c')\in \mathcal{C}$. Since $F$ is a maximal matching and $G$ is 3-uniform, for every edge $e\in E(G)\setminus F$, $|e\setminus X^F|\leq 2$. 			
			
			We define a 2-precoloring extension $(X,c)$ of $c'$ as follows. We define the sets $X_0$, $X_1, \dots$ iteratively. Let $X_0=X^F$. Let $c(v)=c'(v)$ for all $v\in X^F$. Suppose that we have defined $X_i$. If there exists an edge $e\in E(G)$ such that $e\subseteq X_i$ and $e$ is monochromatic, then $c'$ does not have a  2-precoloring extension and we return this determination. If there exists an edge $e\in E(G)$ such that $|e\setminus X_i|=1$ and $c(e\cap X_i)=\{j\}$ for some $j\in[2]$, we define $c(w)$ to be the unique element of $[2]\setminus\{j\}$ for $w\in e\setminus X_i$, and define $X_{i+1}=X_i\cup \{w\}$. Otherwise we stop and let $X=X_i$. This terminates within at most $O(n^3)$ steps. From the construction, clearly $\{(X_i, c)\}$ is equivalent to $\{(X_{i+1}, c)\}$ at every step; and it follows that $\{(X, c)\}$ is equivalent to $\{(X^F, c')\}$, and that if this step returns a determination that $(X^F, c')$ does not 2-extend to $G$, then this determination is correct.
			
			We define a 2-SAT instance as follows. For every $v\in V(G)\setminus X$, we have a variable $x_v$. Let $E'\subseteq E(G)$ be the set of edges such that $|e\setminus X|=2$ for all $e\in E'$. For every edge $e\in E'$, we create a clause $C_e$. Let $e=\{v,u,w\}$ with $v\in X$ and $u,w\in V(G)\setminus X$. If $c(v)=1$, we set $C_e=x_{u}\vee x_{w}$. Otherwise, let $C_e=\overline{x_{u}}\vee \overline{x_{w}}$.
			
			If the 2-SAT instance has a solution $x$, where "true" and "false" are represented by 1 and 0 respectively, then we set $c(v)= x_v+1$ for every $v\in V(G)\setminus X$. Take an edge $e\in E(G)$. If $|e\setminus X|\leq 1$, by the construction of $X$, $e$ is not monochromatic. If $|e\setminus X|=2$, the clause $C_e$ of 2-SAT instance and the construction of $c$ guarantees that at least one of the vertices in $e\setminus X$ receives the opposite color from the vertex in $e\cap X$. Since $F$ is maximal, there is no edge $e$ in $E(G)$ with $|e\setminus X|=3$. Thus, $c$ is a 2-precoloring extension of $(X,c')$. 
			
			If there is a 2-precoloring extension $d$ of $(X,c)$, then we set $x_v=d(v)-1$ for every $v\in V(G)\setminus X$. For every edge $e=\{v,u,w\}\in E'$ with $e\cap X=\{v\}$, if $d(v)=1$, then $C_e= x_{u}\vee x_{w}$. Since $e$ is not monochromatic, without loss of generality we may assume $d(u)=2$, and so $x_u=d(v)-1=1$. Thus, the clause $C_e$ is satisfied. A similar argument applies for $d(v)=2$. From the construction of clauses of this 2-SAT instance, we conclude that $x$ is a solution to the 2-SAT instance.
			
			Therefore, deciding whether $(X,c)$ has a 2-coloring extension is equivalent to solving the 2-SAT instance defined above. 
			
			It remains to show that this can be done in polynomial-time. Let $n$ be the number of vertices of $G$. Constructing the set $X$ takes time $O(n^6)$. Constructing the equivalent 2-SAT instance takes time $O(n^3)$. Solving this 2-SAT instance takes time $O(n)$. So the total running time is $O(n^6)$.
	\end{proof}	
	
	This immediately implies, for fixed $r$, a polynomial-time algorithm for 2-coloring tournaments with no $r$ vertex-disjoint cyclic triangles, which was first proved by Hajebi \cite{Sepehr}. 

\section{Algorithm for the case $s\leq r-1$} \label{Sec-kBddrCol-sBdd}
	In this section, we prove:

	\begin{theorem} \label{kBddrCol-sBdd}
		For fixed positive integers $r,k,s$ with $s\leq r-1$, the \textsc{$k$-Bounded Hypergraph $r$-Precoloring Extension Problem} restricted to hypergraphs $G$ with $\nu(G)\leq s$ is polynomial-time solvable.
	\end{theorem}

	\begin{lemma}\label{GuessS}
		Let $r, k, s \in \mathbb{N}$ with $s \leq r-1$. Let $G$ be a $k$-bounded hypergraph with $\nu(G) \leq s$. Given a partial $r$-coloring $(X,c)$ of $G$, we define $E_{i}=\{e\in E(G):e\cap X\subseteq c^{-1}(i)\}$. If $E_{i}\neq \emptyset$ for all $i\in [r]$, then there is a vertex set $X'\supset X$ and a collection of partial $r$-colorings $\mathcal{C}$ such that
		\begin{itemize} 	
			\item For every $(X^*,c^*)\in \mathcal{C}$, $X^*=X'$;
			\item $|\mathcal{C}|\leq r^{ks}= O(1)$, and $\mathcal{C}$ can be computed from $(X,c)$ in time $O(n^k)$;
			\item There is a color $j\in [r]$ such that for every edge $e\in E_{j}$, $|e\cap X'|\geq |e\cap X|+1$; and
			\item $\mathcal{C}$ is $r$-equivalent to $(X,c)$.
		\end{itemize}
	\end{lemma}
	\begin{proof}
		Let $S$ be a matching in $G$ such that $S \subseteq \bigcup_{i\in [r]} E_i$, and $S$ is maximal with respect to this condition. Let $X^{S}=\cup_{e\in S} e$. Let $X'=X\cup X^{S}$. Let $\mathcal{C}$ be the set of all partial $r$-colorings $(X',c':X'\rightarrow [r])$ such that $c'|_{X}=c$. The first property follows immediately from the construction. Since $|S|\leq s$, $|X^{S}|\leq ks$, and we have $|\mathcal{C}|\leq r^{ks}= O(1)$. Finding $S$ takes time $O(n^k)$, and thus, $\mathcal{C}$ can be computed from $(X,c)$ in time $O(n^k)$. This proves the second property.
				
	    For every $e \in S$, there exists $i \in [r]$ such that $e \in E_i$, and therefore we have that $c(v) = i$ for all $v \in e \cap X$. Since $|S|\leq s\leq r-1$, there exists a color $j\in [r]$ such that $c(v) \neq j$ for all $v \in X \cap X^{S}$. Let $e$ be an edge in $ E_j$. We know that $e\cap X\subseteq c^{-1}(j)$, so $e\cap X^{S}\cap X=\emptyset$. But from the definition of $S$, we have $e\cap X^{S}\neq \emptyset$, as otherwise $S$ is not maximal. Thus, $e\cap (X^{S}\setminus X)\neq \emptyset$. This proves the third property.
		
		Suppose that there is a partial $r$-coloring $(X',c')\in \mathcal{C}$ which $r$-extends to $G$. Then by the construction of $\mathcal{C}$, $c'|_{X}=c$. Thus, every $r$-precoloring extension of $(X',c')$ is also an $r$-precoloring extension of $(X,c)$. Now suppose $(X,c)$ $r$-extends to $V(G)$, that is, there is a coloring $c':V(G)\rightarrow [r]$ with $c'|_{X}=c$. Then by the construction of $\mathcal{C}$, $(X',c'|_{X'})\in \mathcal{C}$. Therefore, the last property holds. 
	\end{proof}
	
	\begin{theorem}
		For fixed positive integers $r,k,s$ with $s\leq r-1$, there is an algorithm with the following specifications:
		\begin{itemize}
			\item Input: A $k$-bounded hypergraph $G$ with $\nu(G)\leq s$, and an $r$-precoloring $(X,c)$.
			\item Output: one of
				\begin{itemize}
					\item an $r$-precoloring extension of $(X,c)$ to $V(G)$;
					\item a determination that $(X,c)$ does not $r$-extend to $G$.
				\end{itemize}
			\item Running time: $O(|V(G)|^k)$.		
		\end{itemize}
	\end{theorem}
	\begin{proof}
		\setcounter{tbox}{0}
		We define a sequence $\mathcal{C}_0, \dots$ of collections of partial $r$-colorings iteratively, as follows. Let $\mathcal{C}_0=\{(X,c)\}$. 
		
		Suppose that we have defined $\mathcal{C}_t$. Given a partial $r$-coloring $(Y,d)\in \mathcal{C}_t$, let $$E_{t,i}^{Y,d}=\{e\in E(G):e\cap Y\subseteq d^{-1}(i)\}.$$ If $E_{t,i}^{Y,d}=\emptyset$ for some $i\in [r]$ and $(Y, d) \in \mathcal{C}_t$, then we define $d'$ by setting $d'|_Y = d|_Y$ and $d'(v)=i$ for all $v\in V(G)\setminus Y$ and return $d'$. Note that $d'$ is an $r$-coloring of $G$: Since $(Y, d)$ is a partial $r$-coloring, it follows that no edge of $G[Y]$ is monochromatic. Therefore, if $G$ contains an edge $e$ which is monochromatic with respect to $d'$, then $e \setminus Y \neq \emptyset$. It follows that $e  \setminus Y \neq \emptyset$, and since $d'(v) = i$ for all $v \in V(G) \setminus Y$, it follows that every vertex of $e$ is colored $i$ by $d'$. But then $e \cap Y \subseteq d^{-1}(i)$, a contradiction. This shows that $d'$ is an $r$-coloring of $G$. 
		
		Otherwise, for every $(Y, d) \in \mathcal{C}_t$, we have that $E_{t,i}^{Y,d} \neq \emptyset$, and so there is a collection of  partial $r$-colorings  $\mathcal{C}_{t+1}^{Y,d}$ which satisfies the properties in Lemma \ref{GuessS} applied to $G$ and $(Y, d)$. Let $\mathcal{C}_{t+1}=\cup_{(Y,d)\in \mathcal{C}_{t}} \mathcal{C}_{t+1}^{Y,d}$. By Lemma \ref{GuessS}, $\mathcal{C}_{t+1}$ is $r$-equivalent to $\mathcal{C}_t$; and inductively, $\mathcal{C}_{t+1}$ is equivalent to $\mathcal{C}_0=\{(X,c)\}$. Thus, if $\mathcal{C}_{t+1}=\emptyset$, then $(X,c)$ does not $r$-extend to $G$ and we return this. 
	
		
		
		
	    It remains to show that this algorithm terminates and runs in polynomial time. To prove this, we define a potential function $\psi((Y,d))=\sum_{i\in [r]} \max (\{0\}\cup \{ |e\setminus Y|:e\cap Y\subseteq d^{-1}(i)\})$. We have $\psi((X,c))\leq rk$ since each summand is at most $k$. We prove by induction on $t$ that for every $(Y,d)\in \mathcal{C}_{t}$, we have $\psi((Y,d))\leq rk-t$. 
		
		It suffices to show that if $(Y, d) \in \mathcal{C}_t$ and $(Y', d') \in \mathcal{C}_{t+1}^{Y,d}$ then $\psi((Y', d')) \leq \psi((Y, d)) -1$. By the third property of Lemma \ref{GuessS}, there is a color $j\in[r]$ such that for every edge $e\in E_{t,j}^{Y, d}$, $|e\cap Y'|\geq |e\cap Y|+1$, which means that $$\max (\{0\} \cup \{|e\setminus Y'|: e\cap Y'\subseteq d'^{-1}(j)\}) \leq \max (\{|e\setminus Y|: e\cap Y\subseteq d^{-1}(j)\})-1.$$ It follows that $\psi((Y', d')) \leq \psi((Y, d)) -1$, as claimed. 
		
		Since $\psi((Y, d)) \geq 0$ for every partial $r$-coloring $(Y, d)$ of $G$, it follows that this algorithm terminates in $t'$ steps for some $t'\leq rk$. Since there are $O(1)$ iterations, and by Lemma \ref{GuessS}, we have $|\mathcal{C}_{t}|=O(1)$ for all $t \leq t'$. Moreover, the set $\mathcal{C}_{t+1}$ can be computed from $\mathcal{C}_t$ in time $|\mathcal{C}_t|\cdot O(n^k)=O(n^k)$. Thus, each step takes time $O(n^k)$. So the total running time is $O(n^k)$. 
	\end{proof}

\section{NP-hardness results for bounded matching number} \label{Sec-NP}
	Let $G$ and $H$ be two hypergraphs. We define an operation, $\ltimes$, via $$G\ltimes H:=(V(G)\cup V(H),E(G)\cup \{e\cup \{x\}: e\in E(H), x\in V(G)\}).$$

	We have the following properties.
	\begin{lemma} \label{BddNu}
		Let $G$ and $H$ be hypergraphs. Then $\nu(G\ltimes H)\leq |V(G)|$.
	\end{lemma}
	\begin{proof}
		This follows immediately from the fact that every edge in $G\ltimes H$ contains at least one vertex in $V(G)$.
	\end{proof}
	
	\begin{lemma} \label{Chi}
		Let $H$ be a hypergraph. If $G$ is a hypergraph with $\chi(G)=r$, then $G\ltimes H$ is $r$-colorable if and only if $H$ is $r$-colorable.
	\end{lemma}
	\begin{proof}
		Suppose for a contradiction that $G\ltimes H$ has an $r$-coloring $c$ and $H$ is not $r$-colorable. Since $c|_{V(H)}$ is not an $r$-coloring of $H$, there exists an edge $e\in E(H)$ such that $e$ is monochromatic with respect to $c|_{V(H)}$. Since $\chi(G)=r$ and $(G \ltimes H)[V(G)] = G$, there exist vertices $v_1,\dots,v_r\in V(G)$ such that $c(v_i)=i$ for all $i \in [r]$. But then one of the edges $e\cup \{v_1\},\dots,e\cup \{v_r\}$ is monochromatic, which contradicts the fact that $c$ is an $r$-coloring of $G\ltimes H$.
		
		Now suppose that $H$ has an $r$-coloring $d$. Since $\chi(G)=r$, $G$ has an $r$-coloring $d'$. We define a new function $d^*:V(G\ltimes H)\rightarrow [r]$ with $d^*(v)=d(v)$ if $v\in V(H)$ and $d^*(v)=d'(v)$ otherwise. For every edge $e\in E(G\ltimes H)$, if $e\in E(G)$, then $d^*|_e=d'|_e$. So $e$ is not monochromatic. Otherwise $e=e'\cup \{v\}$ for some $e'\in E(H)$ and $v\in V(G)$. Then $d^*|_{e'}=d|_{e'}$, and since the edge $e'$ is not monochromatic, it follows that $e$ is not monochromatic. Thus $d^*$ is an $r$-coloring of $G\ltimes H$.
	\end{proof}

	\begin{theorem} \label{NP-Bdd}
		Given fixed integers $k$ and $r$ with $k,r\geq 2$, if the \textsc{$k$-Bounded Hypergraph $r$-Coloring  Problem} is NP-complete, then the \textsc{$(k+1)$-Bounded Hypergraph $r$-Coloring  Problem} restricted to hypergraphs with $\nu(G)\leq r$ is NP-complete.
	\end{theorem}
	\begin{proof}
		Let $H$ be a $k$-bounded hypergraph. We set the hypergraph $G=K_r$, a complete graph on $r$ vertices. We have $\chi(G)=r$. The hypergraph $G\ltimes H$ can be constructed from $G$ and $H$ in time $O(n^{k+1})$, where $n=|V(G\ltimes H)|$. By the construction, $G\ltimes H$ is $(k+1)$-bounded. The remaining part of the proof follows immediately from Lemmas \ref{BddNu} and \ref{Chi}.
	\end{proof}
	
	\begin{theorem} \label{NP-Unf}
		Given fixed integers $k$ and $r$ with $k,r\geq 2$, if the \textsc{$k$-Uniform Hypergraph $r$-Coloring  Problem} is NP-complete, then the \textsc{$(k+1)$-Uniform Hypergraph $r$-Coloring  Problem} restricted to hypergraphs with $\nu(G)\leq (r-1)k+1$ is NP-complete.
	\end{theorem}
	\begin{proof}
		Let $H$ be a $k$-uniform hypergraph and let $G$ be the complete $(k+1)$-uniform hypergraph with $(r-1)k+1$ vertices. The hypergraph $G\ltimes H$ can be constructed from $G$ and $H$ in time $O(n^{k+1})$, where $n=|V(G\ltimes H)|$. By the construction, $G\ltimes H$ is $(k+1)$-uniform. 
		
		We want to show that $\chi(G)=r$. We choose $k$ vertices to color $i$ for every $i\in [r-1]$, and color the remaining vertex $r$. Since $G$ is $(k+1)$-uniform, every edge of $G$ receives at least two colors. Thus, $\chi(G)\leq r$. Suppose for a contradiction that $\chi(G)\leq r-1$. Then take an $(r-1)$-coloring $c$ of $G$. There exists one color $i$ with $|c^{-1}(i)|\geq \lceil  \frac{(r-1)k+1}{r-1} \rceil \geq \lceil  k+\frac{1}{r-1} \rceil =k+1$. This means that there is a monochromatic edge in $G$, which contradicts  the fact that $c$ is an $(r-1)$-coloring of $G$.
		
		The remaining part of the proof follows immediately from Lemmas \ref{BddNu} and \ref{Chi}.
	\end{proof}

	\begin{theorem} \label{thm:precoloring-ext}
		Given fixed integers $k$ and $r$ with $k,r\geq 2$, if the \textsc{$k$-Uniform Hypergraph $r$-Coloring Problem} is NP-complete, then the \textsc{$(k+1)$-Uniform Hypergraph $r$-Precoloring Extension Problem} restricted to hypergraphs with $\nu(G)\leq r$ is NP-complete.
	\end{theorem}
	\begin{proof}
		Let $H$ be a $k$-uniform hypergraph and let $G$ be a graph with a set of vertices $\{v_1,\dots,v_r\}$ and no edges. Define the precoloring of $G\ltimes H$ to be $(V(G),c')$ with $c'(v_i)=i$ for all $i \in [r]$. The hypergraph $G\ltimes H$ can be constructed from $G$ and $H$ in time $O(n^{k+1})$, and the precoloring $(V(G),c')$ of $G\ltimes H$ can be constructed in time $O(n)$, where $n=|V(G\ltimes H)|$. The graph $H$ is $k$-uniform and $E(G)=\emptyset$, so $G\ltimes H$ is $(k+1)$-uniform.
		
		It remains to show that $G\ltimes H$ has an $r$-precoloring extension with respect to the precoloring $(V(G),c')$ if and only if $H$ is $r$-colorable. 
		
		Suppose $G\ltimes H$ has an $r$-precoloring extension $c$. Assume for a contradiction that $H$ is not $r$-colorable. Since $c|_{V(H)}$ is not an $r$-coloring of $G$, there exists an edge $e\in E(H)$ such that $e$ is monochromatic. By the definition of $c'$, one of the vertices $ v_1,\dots,v_r$ receives the same color as $e$, which contradicts the fact that $c$ is an $r$-precoloring extension of $G\ltimes H$ and $(V(G), c')$.
		
		Now suppose that $H$ has an $r$-coloring $d$. We define a new function $d^*:V(G\ltimes H)\rightarrow [r]$ with $d^*(v)=d(v)$ if $v\in V(H)$ and $d^*(v)=c'(v)$ otherwise. For every edge $e\in E(G\ltimes H)$, $e=e'\cup \{v\}$ for some $e'\in E(H)$ and $v\in V(G)$. Then $d^*|_{e'}=d|_{e'}$. The edge $e'$ is not monochromatic, so $e$ is not monochromatic. Thus $d^*$ is an $r$-coloring of $G\ltimes H$ which $r$-extends $(V(G), c')$.		
	\end{proof}
	
	\begin{theorem}\label{kUnfrColNP}
		Given fixed integers $k$ and $r$ with $k,r\geq 2$, the \textsc{$k$-Uniform Hypergraph $r$-Coloring Problem} is NP-complete if $k+r\geq 5$.
	\end{theorem}
	\begin{proof}
		The statement holds for the cases $k=3$ and $r=2$ by Theorem \ref{3Unf2ColNP}, and $k=2$ and $r\geq 3$ by Theorem \ref{rColNP}. By Theorem \ref{NP-Unf}, if the \textsc{$k$-Uniform Hypergraph $r$-Coloring Problem} is NP-complete, then the \textsc{$(k+1)$-Uniform Hypergraph $r$-Coloring Problem} is NP-complete.
	\end{proof}

	Now we are ready to prove our main results.
	\begin{proof}[Proof of Theorem \ref{NP-Dic}]
		The first and second polynomial-time solvable cases follow from Theorem \ref{kBddrCol-sBdd} and Theorem \ref{3Bdd2Col} respectively. The third polynomial-time solvable case follows from Theorem \ref{rP2}, as a graph $G$ with $\nu(G)\leq s$ is guaranteed to be $(s+1)P_2$-free. Combining Theorem \ref{kUnfrColNP} with either Theorem  \ref{NP-Bdd} or Theorem \ref{thm:precoloring-ext}, we have completed the dichotomies.
	\end{proof}

	\begin{proof}[Proof of Theorem \ref{NP-Gap}]
		The first and second polynomial-time solvable cases follow from Theorem \ref{kBddrCol-sBdd} and Theorem \ref{3Bdd2Col} respectively. The third polynomial-time solvable case follows from Theorem \ref{rP2}, as a graph $G$ with $\nu(G)\leq s$ is guaranteed to be $(s+1)P_2$-free. The NP-completeness result comes from Theorems \ref{kUnfrColNP} and \ref{NP-Unf}.
	\end{proof}	

\section{Stable Set} \label{Sec-StbSet}
	
	
	In this section, we consider the complexity of stable set problems in hypergraphs with bounded matching number. The \textsc{$k$-Uniform Hypergraph Maximum Weight Stable Set Problem} is the following: Given a $k$-uniform hypergraph $G$ and a weight function $w:V(G)\rightarrow \mathbb{R}_{\geq 0}$, compute a stable set $S\subseteq V(G)$ with $w(S)$ maximized. When all weights are 1, this is called the \textsc{$k$-Uniform Hypergraph Maximum Stable Set Problem}.
	
	For graphs, the \textsc{Graph Maximum Weight Stable Set Problem} can be solved in polynomial time if the maximum size of an induced matching is bounded:
	\begin{theorem}[Balas and Yu \cite{MaxStb-IndMatching}]
		For a fixed positive integer $s$, the \textsc{Graph Maximum Weight Stable Set Problem} restricted to $sP_2$-free graphs can be solved polynomial time.
	\end{theorem}

	For hypergraphs, we notice that:
	\begin{theorem} \label{MSS}
		For fixed positive integers $k$ and $s$, the \textsc{$k$-Uniform Hypergraph Maximum Stable Set Problem} restricted to hypergraphs with $\nu(G)\leq s$ is polynomial-time solvable.
	\end{theorem}
	\begin{proof}
		Let $G$ be a $k$-uniform hypergraph with $\nu(G) \leq s$, and let $F\subseteq E(G)$ be a maximal matching. Let $n$ be the number of vertices of $G$. We have $|F|\leq s$. The set $V(G)\setminus (\cup _{e\in F} e)$ is stable as $F$ is maximal, and $|V(G)\setminus (\cup _{e\in F} e)|\geq n-ks$. Thus, a maximum stable set of $G$ is of size at least $n-ks$.
		
		Therefore, to find a maximum stable set, we can simply enumerate all choices of a set $U\subseteq V(G)$ with $|U|\leq ks$, and check if the set $V(G)\setminus U$ is stable, and return the largest stable set found this way. There are $n^{ks}$ choices of the set $U$, and for each $U$, it takes time $O(n^k)$ to verify stability. Thus, the total running time is $O(n^{k(s+1)})$.
	\end{proof}

    In contrast, we show the following result for the weighted version of the problem: 

	\begin{theorem} \label{MWSS}
		For a fixed positive integer $k\geq 3$, the \textsc{$k$-Uniform Hypergraph Maximum Weight Stable Set Problem} restricted to hypergraphs with $\nu(G)\leq 1$ is NP-complete.
	\end{theorem}
	
	In order to prove Theorem \ref{MWSS}, we need the following results.
	\begin{theorem} [Garey and Johnson \cite{MaxStb-NP}] \label{MSS-NP}
		The \textsc{Maximum Stable Set Problem} is NP-complete.
	\end{theorem}
	
	\begin{lemma} \label{MWSS-Ind}
		For a fixed positive integers $k\geq 3$, if the \textsc{$(k-1)$-Uniform Hypergraph Maximum Weight Stable Set Problem} is NP-complete, then the \textsc{$k$-Uniform Hypergraph Maximum Weight Stable Set Problem} restricted to hypergraphs with $\nu(G)\leq 1$ is NP-complete.
	\end{lemma}
	\begin{proof}
		Suppose the \textsc{$(k-1)$-Uniform Hypergraph Maximum Weight Stable Set Problem} is NP-complete. Let $G$ be a $(k-1)$-uniform hypergraph with weight function $w$. We construct a new $k$-uniform hypergraph $H$ with $V(H)=\{v\}\cup V(G)$ and $E(H)=\{\{v\}\cup e:e\in E(G)\}$. We define the weight function $w':V(G)\rightarrow \mathbb{R}_{\geq 0}$ such that $w'(u)=w(u)$ for each $u\in V(G)$ and $w'(v)=\sum_{u\in V(G)}w(u)+1$. From the construction, since $v$ is contained in every edge of $H$, it follows that the hypergraph $H$ satisfies $\nu(H)\leq 1$.
		
		For a set $T \subseteq V(G)$, $T$ is a stable set of $G$ if and only if $T \cup \{v\}$ is a stable set of $H$. Let $S$ be a maximum weight stable set of $H$ with respect to the weight function $w'$. By the construction, the vertex $v$ is in $S$.  It follows that $S \setminus \{v\}$ is a maximum weight stable set of $G$, and thus, to find a maximum weight stable set of $G$, it suffices to find a maximum weight stable set of $H$. 
		
		
		
		Since the construction can be done in polynomial time, we have proved this lemma.
	\end{proof}
	
	\begin{proof} [Proof of Theorem \ref{MWSS}]
		We prove this by induction on $k$. When $k=2$, by Theorem \ref{MSS-NP}, the \textsc{Graph Maximum Stable Set Problem} is NP-complete. Thus, the \textsc{Graph Maximum Weight Stable Set Problem} is NP-complete.
		
		Suppose that the \textsc{$k$-Uniform Hypergraph Maximum Weight Stable Set Problem} is NP-complete. By Lemma \ref{MWSS-Ind}, the \textsc{$(k+1)$-Uniform Hypergraph Maximum Weight Stable Set Problem} restricted to hypergraphs with $\nu(G)\leq 1$ is NP-complete. Moreover, the \textsc{$(k+1)$-Uniform Hypergraph Maximum Weight Stable Set Problem} is NP-complete.
	\end{proof}

\section{Excluding an induced subhypergraph with one edge}\label{Sec-OneEdge}
	For $t \in \mathbb{N}$ with $t \geq 3$, let $H_t$ be the 3-bounded hypergraph with $t+3$ vertices and one edge. In this section, we give a polynomial-time algorithm for 2-coloring 3-bounded $H_t$-free hypergraphs.
	
	\begin{lemma} \label{Ind-guess}
	    Let $t \in \mathbb{N}$, and let $G$ be a $3$-bounded $H_t$-free hypergraph. There is a polynomial-time algorithm to test if $G$ has a 2-coloring with at least $t$ vertices of each color.
	\end{lemma}
	\begin{proof}	
	    \setcounter{tbox}{0}
        We may assume that $|e| \geq 2$ for all $e \in E(G)$, since $G$ is not 2-colorable otherwise. Let $\mathcal{C}$ be a partial 2-coloring collection containing all partial $2$-colorings $(X\cup Y,c')$ for every pair of disjoint sets $X,Y\subseteq V(G)$ with $|X|=|Y|=t$ and $X$, $Y$ stable, and $c':X\cup Y\rightarrow [2]$ with $c'(v)=1$ for all $v\in X$ and $c'(v)=2$ otherwise.  It suffices to show that $\mathcal{C}$ has the following three properties.
        
        \tbox{ \label{eq:equiv} $\mathcal{C}$ is 2-equivalent to the collection of all 2-colorings of G with at least $t$ vertices of each color. }
		
		We only need to show that if $G$ has a 2-coloring $c$ with at least $t$ vertices of each color, then there exists a partial 2-coloring in $\mathcal{C}$ which 2-extends to $G$. Let $c$ be a 2-coloring of $G$ such that $|c^{-1}(i)|\geq t$ for all $i\in [2]$. 
		Let $X$ and $Y$ be subsets of $c^{-1}(1)$ and $c^{-1}(2)$ respectively, with $|X|=|Y|=t$. We have $(X\cup Y, c|_{X\cup Y})\in \mathcal{C}$, and $c$ is a 2-precoloring extension of $(X \cup Y, c|_{X\cup Y})$ to $V(G)$. This proves \eqref{eq:equiv}. 
		
		\tbox{ \label{eq:poly} $\mathcal{C}$ can be computed in time $O(n^{2t+3})$, where $n=|V(G)|$.}

		Since $G$ is 3-bounded, $|E(G)|\leq O(n^3)$. By construction, we have $|\mathcal{C}|\leq O(n^{2t})$. Constructing the sets $X$, $Y$ and the corresponding partial 2-coloring $c$ takes time $O(n^{2t})$. Checking whether $(X \cup Y,c)$ is a partial 2-coloring takes time $O(n^3)$. Thus, $\mathcal{C}$ can be constructed in time $O(n^{2t+3})$. This proves \eqref{eq:poly}. 
        
        \tbox{ \label{eq:2sat} For every partial 2-coloring $(X\cup Y, c)$ in $\mathcal{C}$,  whether $c$ 2-extends to $G$ can be decided in polynomial time.}
        
		
		For convenience, let us denote $S=X\cup Y$. We define a 2-SAT instance as follows. For every $v\in V(G)\setminus S$, we have a variable $x_v$. Let $E'\subseteq E(G)$ be the set of edges $e\in E(G)$ with $|c(e \cap S)|=1$. Note that for every edge $e \in E'$, we have $e \cap S \neq \emptyset$ and $e \setminus S \neq \emptyset$ (since $(S, c)$ is a partial 2-coloring). Thus, $|e \setminus S| \in \{1, 2\}$. For every edge $e\in E'$, we create a clause $C_e$. Let $u, w \in e \setminus S$ with $u \neq w$ with $|e \setminus S| = 2$. If $c(e \cap S)=\{1\}$, we set $C_e=x_{u}\vee x_{w}$. Otherwise, let $C_e=\overline{x_{u}}\vee \overline{x_{w}}$. Next, let $E''$ be the set of edges $e \in E(G)$ with $|e| = 2$ and $e \cap S = \emptyset$. For every $e \in E''$, say $e = \{u, w\}$, we add two clauses $C'_e=x_{u}\vee x_{w}$ and $C''_e=\overline{x_{u}}\vee \overline{x_{w}}$. 
		
		If the 2-SAT instance has a solution $(s_v)_{v \in V(G) \setminus S}$, where "true" and "false" are represented by 1 and 0 respectively, then we set $d(v)= s_v+1$ for every $v\in V(G)\setminus S$, and $d(v) = c(v)$ for all $v \in S$. We claim that $d$ is a 2-coloring of $G$. Consider an edge $e\in E(G)$. If $|c(e \cap S)| > 1$, then $e$ is not monochromatic. If $|c(e \cap S)| = 1$, then $e \in E'$. It follows that the clause $C_e$ of 2-SAT instance and the construction of $d$ guarantees that at least one vertex in $e\setminus S$ receives the opposite color from the vertices in $e\cap S$. Since both sets are non-empty, it follows that $e$ is not monochromatic. It remains to consider the case that $e \cap S = \emptyset$. If $|e| = 2$, then the clauses $C'_e$ and $C''_e$ guarantee that the two vertices of $e$ receive different colors. Therefore, we may assume that $|e| = |e \setminus S| = 3$. Suppose for a contradiction that $e$ is monochromatic. Without loss of generality, assume $d(v)=1$ for all $v\in e$. Let $X = (S \cap c^{-1}(1))$, and consider the set $X\cup e$. Since all edges with a non-empty intersection with $S$ and all edges of size 2 are non-monochromatic, there is no edge $e'\in E(G)$ with $e'\subseteq X\cup e$ and $e' \neq e$. Thus, $G[X\cup e]$ is an induced copy of $H_t$ in $G$, which contradicts the fact that $G$ is $H_t$-free. Therefore, $d$ is a 2-precoloring extension of $(S,c)$. 
		
		If there is a 2-precoloring extension $d$ of $(S, c)$, then we set $x_v=d(v)-1$ for every $v\in V(G)\setminus S$. For every edge $e \in E'$, if $C_e= x_{u}\vee x_{w}$, then $e \cap S$ contains only vertices colored 1, and so $d(u) = 2$ or $d(w) =2$; it follows that $C_e$ is satisfied. If $C_e= \overline{x_{u}}\vee \overline{x_{w}}$, then $e \cap S$ contain only vertices colored 2, and so $d(u) = 1$ or $d(w) =1$; it follows that $C_e$ is satisfied. For every edge $e = \{u, w\} \in E''$, it follows that $d(u) \neq d(w)$, and hence one of $x_u, x_w$ is "true" and the other is "false." It follows that $C'_e$ and $C''_e$ are satisfies. 	From the construction of clauses of this 2-SAT instance, we conclude that this assignment is a solution to the 2-SAT instance. Therefore, deciding whether $(S, c)$ has a 2-coloring extension is equivalent to solving the 2-SAT instance defined above. 
		
		It remains to show that this can be done in polynomial time. Constructing the 2-SAT instance takes time $O(n^3)$. Solving this 2-SAT instance takes time $O(n^3)$. So the total running time is $O(n^3)$. This proves \eqref{eq:2sat} and concludes the proof. 
	\end{proof}	

	\begin{theorem} \label{thm:htmain}
		Let $t \in \mathbb{N}$, and let $G$ be a $3$-bounded $H_t$-free hypergraph. There is a polynomial-time algorithm which takes $G$ as input, and outputs either a 2-coloring of $G$, or a determination that $G$ is not 2-colorable.
	\end{theorem}
	\begin{proof}
		If $G$ satisfies the conditions of Lemma \ref{Ind-guess}, then we are done. Otherwise we can go through every possible coloring such that less than $t$ vertices receive color $i$ for some $i\in[2]$, and check whether it is a 2-coloring, in time $O(n^{t+3})$.
	\end{proof}
	
	Note that the proof of Lemma \ref{Ind-guess} can be modified to work for the precoloring extension version of the problem, and so can Theorem \ref{thm:htmain}. 
	
\section{Linear Hypergraphs} \label{sec:linear}
\subsection{The polynomial-time algorithm}

	We say a $k$-uniform hypergraph is \emph{complete} if its edge set is the set of all $k$-vertex subsets of its vertex set. The \emph{hypergraph Ramsey number}, $R_k(n_1,\dots,n_t)$, is the smallest integer $N$ such that for every function $f:E(G)\rightarrow [t]$ for a complete $k$-uniform hypergraph $G$ with at least $N$ vertices, there exists $i\in [t]$ and a set $S\subseteq V(G)$ with $|S|\geq n_i$ such that all edges $e \subseteq S$ satisfy $f(e)=i$. 
	
	\begin{theorem}[Ramsey \cite{Ramsey}] \label{Ramsey}
		For all positive integers $k$, $n_1$, \dots, $n_t$, the hypergraph Ramsey number $R_k(n_1,\dots,n_t)$ exists.
	\end{theorem}

	\begin{lemma} \label{lem:hyperramsey}
		For every positive integer $s$, there exists a positive integer $s'$ such that every 3-uniform linear hypergraph $G$ which contains a matching of size $s'$ contains an induced matching of size $s$.
	\end{lemma} 
	\begin{proof}		
	    We may assume that $s \geq 4$. Let $\{G_1, \dots, G_t\}$ be the set of all linear 3-uniform hypergraphs with vertex set $\{x_1, \dots, x_9\}$. Since there at most $2^{9 \choose 3}$ distinct 3-uniform (labelled) hypergraphs on 9 vertices, it follows that $t \leq 2^{9 \choose 3}$. Let $s' = R_3(n_1, \dots, n_t)$ with $n_1 = \dots = n_t = s$. 
	    
		Let $\{e_1, \dots, e_{s'}\}$ be a matching of size $s'$ in $G$. For $i\in [s']$, let $e_i = \{u_i, v_i, w_i\}$. Let $H$ be a complete 3-uniform hypergraph $V(H) = \{1, \dots, s'\}$. We define $f: E(H) \rightarrow [t]$ as follows. For $e = \{i, j, k\} \subseteq [s']$ with $i<j<k$, we define $f(e) = m$ if $G[e_i \cup e_j \cup e_k]$ is isomorphic to $G_m$ via the isomorphism $u_i \mapsto x_1$, $v_i \mapsto x_2$, $w_i \mapsto x_3$, $u_j \mapsto x_4$, $v_j \mapsto x_5$, $w_j \mapsto x_6$, $u_k \mapsto x_7$, $v_k\mapsto x_8$ and $w_k \mapsto x_9$.

		
		
		
		From Theorem \ref{Ramsey}, it follows that there is a set $S \subseteq [s']$ with $|S| = s$ and $m \in [t]$ such that $f(e) = m$ for all $e \subseteq S$. We claim that $$E(G_m) = \{\{x_1, x_2, x_3\}, \{x_4, x_5, x_6\}, \{x_7, x_8, x_9\}\}.$$ 
		Let $i, j, k, l \in S$ with $i < j < k < l$. Since $G[e_i, e_j, e_k]$ contains the edges $e_i, e_j, e_k$, it follows that $\{\{x_1, x_2, x_3\}, \{x_4, x_5, x_6\}, \{x_7, x_8, x_9\}\} \subseteq E(G_m)$. Suppose for a contradiction that $E(G_m)$ contains a fourth edge $\{x_a, x_b, x_c\}$. Then, since $G_m$ is linear, we may assume that $a \in \{1,2,3\}$, $b \in \{4,5,6\}$, and $c \in \{7,8,9\}$. The graphs $G[e_i \cup e_j \cup e_k]$ and $G[e_i \cup e_j \cup e_l]$ are isomorphic to $G_m$ via isomorphisms $\varphi, \varphi'$, say; and from the definition of $f$ it follows that $\varphi^{-1}(x_a) = \varphi'^{-1}(x_a)$, $\varphi^{-1}(x_b) = \varphi'^{-1}(x_b)$, and $\varphi^{-1}(x_c) \neq \varphi'^{-1}(x_c)$ (since $\varphi^{-1}(x_c) \in e_k$ and $\varphi'^{-1}(x_c) \in e_l$ and $e_k \cap e_l = \emptyset$). But this implies that $G$ contains the edges $\varphi^{-1}(\{x_a, x_b, x_c\})$ and $\varphi'^{-1}(\{x_a, x_b, x_c\})$ which have exactly two vertices in common, contrary to the assumption that $G$ is linear. This proves our claim. 
		
		It follows that $G\left[\bigcup_{q \in S} e_q\right]$ is an induced matching of size $s$ in $G$. 
	\end{proof}

	\begin{theorem}
		For all $s$, the \textsc{2-Precolouring Extension Problem} restricted to 3-uniform linear hypergraphs with no induced matching of size at least $s$ can be solved in polynomial time.
	\end{theorem}
	\begin{proof}
		By Theorem \ref{3Bdd2Col} and Lemma \ref{lem:hyperramsey}.
	\end{proof}
	
	Note that from Theorem \ref{kBddrCol-sBdd} and Lemma \ref{lem:hyperramsey}, it also follows that for all $s$ there exists $s'$ such that for all $S > s'$, the \textsc{$S$-Precolouring Extension Problem} restricted to 3-uniform linear hypergraphs with no induced matching of size at least $s$ can be solved in polynomial time. 
	

\subsection{NP-hardness of 3-coloring with bounded matching number}
	In this section, we prove the following result. 
	\begin{theorem}\label{Linear-3ColNP}
		The \textsc{3-Uniform Hypergraph 3-coloring Problem} restricted to linear hypergraphs with $\nu (G)\leq s$ is NP-complete for all $s\geq 532$.
	\end{theorem}

	We use the following theorems.
	\begin{theorem}[Garey, Johnson and Stockmeyer \cite{3ColBddDeg-NP}] \label{thm:3colbdd}
		The \textsc{3-Coloring Problem} restricted to graphs with maximum degree at most $4$, is NP-complete.
	\end{theorem}

	\begin{theorem} [Vizing \cite{Vizing}, Misra and Gries \cite{Vizing-Algo}] \label{thm:vizing}
		There is a $O(mn)$-algorithm for edge-coloring a graph $G$ with $D + 1$ colors, where $D$ is the maximum degree of $G$, $m$ is the number of edges and $n$ is the number of vertices.
	\end{theorem}
	
	Let us introduce a new way to describe 3-uniform hypergraphs. Instead of using edges with three vertices, we use 2-edges with labeled vertices. Given a graph $G$, we say a function $l:E(G)\rightarrow V(G)$ with $l(e)\notin e$ for all $e\in E(G)$ is a \emph{labeling} of $G$. The vertex $l(e)$ is called the \emph{label} of $e$, and the edge $e$ is a \emph{labeled edge}. 
	
	For a linear 3-uniform hypergraph $G$, let $l:E(G)\rightarrow V(G)$ be a function with $l(e)\in e$ for all $e\in E(G)$. Let $G'$ be the graph with vertex set $V(G)$ and edge set $\{\{e\setminus \{l(e)\}:e\in E(G)\}$, and let $l'(e\setminus \{l(e)\}) = l(e)$. Since $G$ is linear, each edge of $G'$ corresponds to a unique edge of $G$, and thus $l'$ is well-defined. We call $(G',l')$ a \emph{labeled graph representation} of $G$. Notice that with a labeled graph representation, we can reconstruct the corresponding linear 3-uniform hypergraph.
	
	In this section, all of the figures of 3-uniform hypergraphs are drawn using the labeled graph representation. 
	
	The following two lemmas give constructions for gadgets we use in our NP-hardness reduction. The existence of similar gadgets in 3-uniform linear hypergraphs was first proved in \cite{Linear&gadget}. Here we give an explicit construction to obtain a precise bound for the matching number. The construction is shown in Figure \ref{Picture-Gadget1}. 
	\begin{lemma} \label{Gadget1}	
		There is a linear 3-uniform hypergraph $G_1$ with three specified vertices $a, b, c$ with the following properties:
		\begin{itemize}
			\item For every 3-coloring $f$ of $G_1$, either $f(a), f(b), f(c)$ are all distinct, or $f(a)=f(b)=f(c)$.
			\item There is a 3-coloring $f'$ of $G_1$ with $f'(a), f'(b), f'(c)$ all distinct.
			\item There is a set $Z \subseteq V(G_1)$ with $|Z|\leq 19 $ such that $G_1 \setminus Z$ has no edges, and $a,b,c\in Z$.
			\item No edge $e$ of $G_1$ contains more than one of the vertices $a,b,c$.
		\end{itemize}
	\end{lemma}
	\begin{proof}
		\setcounter{tbox}{0}
		We want to define $G_1$ using the labeled graph representation $(G_1',l)$. First, we create three vertices $a,b,c$. Then we create 4 copies of $K_4$, say $H_1, H_2, H_3, H_4$. For $i\in [4]$, let $V(H_i)=\{s_i,t_i,u_i,v_i\}$. We define the labeling $l(s_it_i)= l(u_iv_i)= a$, $l(s_iu_i)=l(t_iv_i)=b$, and $l(s_iv_i)=l(t_iu_i)=c$.
		
		Let $S=V(H_1)\times V(H_2)\times V(H_3)\times V(H_4)$. For every 4-tuple $T=(x,y,z,w) \in S$, we create 5 new copies of $K_4$, say $H_0^T, H_1^T, H_2^T, H_3^T, H_4^T$. Let $V(H_i^T)=\{s_i^T, t_i^T, u_i^T, v_i^T\}$ for $i\in [4]$, and $V(H_0^T)=\{r_1^T, r_2^T, r_3^T, r_4^T\}$. We define the labeling $l(s_i^T t_i^T)= l(u_i^T v_i^T)= a$, $l(s_i^T u_i^T)= l(t_i^T v_i^T)=b$ and $l(s_i^T v_i^T)= l(t_i^T u_i^T)=c$ for $i\in [4]$, and $l(r_1^T r_2^T)= l(r_3^T r_4^T)=a$, $l(r_1^T r_3^T)= l(r_2^T r_4^T)=b$ and $l(r_1^T r_4^T)= l(r_2^T r_3^T)=c$. For each $i\in [4]$, we add edges $s_i^T r_i^T$ with $l(s_i^T r_i^T)=x$, $t_i^T r_i^T$ with $l(t_i^T r_i^T)=y$, $u_i^T r_i^T$ with $l(u_i^T r_i^T)=z$ and $v_i^T r_i^T$ with $l(v_i^T r_i^T)=w$.
		
		Let $V(G_1')=\{a,b,c\}\cup (\cup_{i\in [4]} V(H_i))\cup (\cup_{T\in S} \cup_{i=0}^4 V(H_i^T))$, and $E(G_1')$ be the set of all labeled edges defined above. By the construction, the function $l$ defined above is a labeling of $G_1'$. Notice that there is no edge incident to more than one of the vertices $a,b,c$, and $l(V(G_1'))=\{a,b,c\}\cup (\cup_{i\in [4]} V(H_i))$. Thus, by taking $Z=l(V(G_1'))$, we have $|Z|\leq 19 $ and $a,b,c\in Z$; so $Z$ satisfies the third property of the lemma. We now prove the other properties. 
		
		\setcounter{tbox}{0}		
		\tbox{\label{eq:linear}The 3-uniform hypergraph $G_1$ is linear.}
		    
		    Let $X_1 = \{a, b, c\}$, $X_2 = (\cup_{i\in [4]} V(H_i))$ and $X_3 = (\cup_{T\in S} \cup_{i=0}^4 V(H_i^T))$. From the construction, it follows that for every edge $e$ of $G_1$, there exist $i, j \in [3]$ with $i < j$ such that $e$ contains one vertex of $X_i$ and two vertices of $X_j$ and with $e \cap X_j \in E(G_1')$ (and therefore, $\{l(e \cap X_j)\} = e \cap X_i$). 
		   
		    Suppose for a contradiction that there exist distinct $e, e' \in E(G_1)$ with $|e \cap e'| = 2$. Let $j, j' \in [3]$ such that $|e \cap X_j| = 2$ and $|e' \cap X_{j'}| = 2$. It follows that $j = j'$. Since $G_1'$ is simple, we have $e \cap X_j \neq e' \cap X_j$, and so $e \setminus X_j = e' \setminus X_j = \{l(e \cap X_j)\} = \{l(e' \cap X_j)\}$. But in $G_1'$, every two edges with the same label are not incident to a common vertex, a contradiction. We conclude that $G_1$ is linear. This proves \eqref{eq:linear}. 
		
		\tbox{\label{eq:distinct}There is a 3-coloring $f'$ of $G_1$ with $f'(a), f'(b), f'(c)$ all distinct.}

			We define a function $f':(V(G_1))\rightarrow [3]$ as follows. Let $f'(a)=1$, $f'(b)=2$ and $f'(c)=3$. 			
			For each $i\in[4]$, let $f'(s_i)= f'(t_i)=2$ and $f'(u_i)=f'(v_i)=3$. Since $l(s_it_i)=l(u_iv_i)=a$, the edges of $G_1$ corresponding to labeled edges in $G_1'[V(H_i)]$ are not monochromatic.
			
			For each $T\in S$ and each $i\in[4]$, let $f'(s_i^T)=f'(u_i^T)=1$,  $f'(t_i^T)=f'(v_i^T)=3$, $f'(r_1^T)=f'(r_4^T)=1$, and $f'(r_2^T)=f'(r_3^T)=2$. 
			For $i\in [4]$, no vertex $v\in V(H_i^T)$ has $f'(v)=2$, and no edge between $V(H_i^T)$ and $V(H_0^T)$ is labeled $a$. So there is no monochromatic edge $e$ in $G_1$ with $e \cap \cup_{i=0}^4 V(H_i^T) \neq \emptyset$. Therefore, the function $f'$ is a 3-coloring of $G_1$. This proves \eqref{eq:distinct}. 
			
		\tbox{\label{eq:equal}For each 3-coloring $f$ of $G_1$, either $f(a), f(b), f(c)$ are all distinct, or $f(a)=f(b)=f(c)$.}

			Assume for a contradiction that, without loss of generality, there is a 3-coloring $f$ of $G_1$ such that $f(a)=f(b)$. Without loss of generality, we may assume that $f(a)=f(b)=1$ and $f(c)=2$.
			
			We claim that there exists $x_0\in V(H_1)$ such that $f(x_0)=3$. Assume for a contradiction that every vertex $v\in V(H_1)$ has $f(v)\neq 3$. Since $l(s_1v_1)=c$ and $f(c)=2$, without loss of generality let  $f(s_1)\neq 2$. So $f(s_1)=1$. Since $l(s_1t_1)=a$, $l(s_1u_1)=b$ and $f(a)=f(b)=1$, we have $f(t_1)=f(u_1)=2$. But the edge $t_1u_1$ is labeled $c$ and $f(c)=2$, the corresponding edge $\{t_1,u_1,c\}$ of $G_1$ is monochromatic, which violates the condition that $f$ is a 3-coloring of $G_1$. 
			
			A similar argument holds for every $H_i$ with $i\in \{2,3,4\}$, and $H_j^T$ with $T\in S$ and $j\in \{0,1,\dots,4\}$. There exist vertices $y_0\in V(H_2)$, $z_0\in V(H_3)$, $w_0\in V(H_4)$ such that $f(y_0)=f(z_0)=f(w_0)=3$. Let $T=(x_0,y_0,z_0,w_0)$.
			By the argument above, there is a $j \in [4]$ such that $f(r_j) = 3$. 	Since there is a vertex $v\in V(H_j^T)$ with $f(v)=3$, and $f(l(vr_j))=3$ (because $l(vr_j) \in \{x_0,y_0,z_0,w_0\}$), the edge $\{v,r_j,l(vr_j)\}$ of $G_1$ is monochromatic, which contradicts the condition that $f$ is a 3-coloring of $G_1$. This proves \eqref{eq:equal}. 		
	\end{proof}
	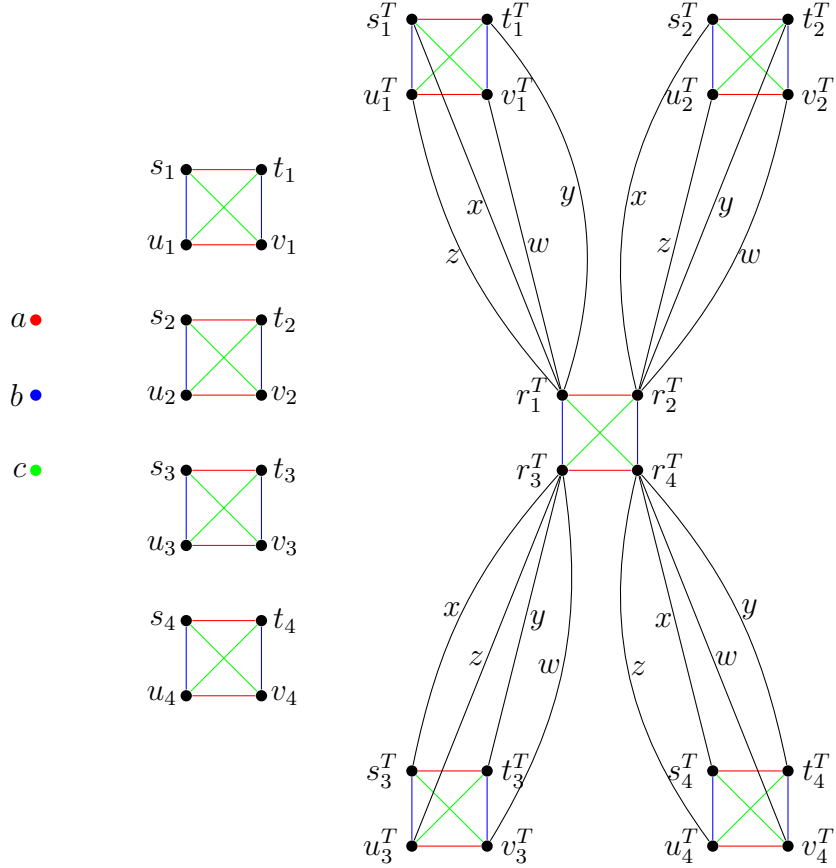
\begin{figure}[!hptb]
		\centering 
		\begin{tikzpicture}[scale=1]
		
		\node [label=left: $a$, red] (a) at (0, 5){};
		\node [label=left: $b$, blue] (b) at (0, 4){};
		\node [label=left: $c$, green] (c) at (0, 3){};
		
		\node [label=left: $s_1$] (s1) at (2, 7){};
		\node [label=right: $t_1$] (t1) at (3, 7){};
		\node [label=left: $u_1$] (u1) at (2, 6){};
		\node [label=right: $v_1$] (v1) at (3, 6){};
		\draw [red] (s1) -- (t1) ;
		\draw [red] (u1) -- (v1) ;
		\draw [blue] (s1) -- (u1) ;
		\draw [blue] (t1) -- (v1) ;
		\draw [green] (s1) -- (v1) ;
		\draw [green] (t1) -- (u1) ;
		
		\node [label=left: $s_2$] (s2) at (2, 5){};
		\node [label=right: $t_2$] (t2) at (3, 5){};
		\node [label=left: $u_2$] (u2) at (2, 4){};
		\node [label=right: $v_2$] (v2) at (3, 4){};
		\draw [red] (s2) -- (t2) ;
		\draw [red] (u2) -- (v2) ;
		\draw [blue] (s2) -- (u2) ;
		\draw [blue] (t2) -- (v2) ;
		\draw [green] (s2) -- (v2) ;
		\draw [green] (t2) -- (u2) ;
		
		\node [label=left: $s_3$] (s3) at (2, 3){};
		\node [label=right: $t_3$] (t3) at (3, 3){};
		\node [label=left: $u_3$] (u3) at (2, 2){};
		\node [label=right: $v_3$] (v3) at (3, 2){};
		\draw [red] (s3) -- (t3) ;
		\draw [red] (u3) -- (v3) ;
		\draw [blue] (s3) -- (u3) ;
		\draw [blue] (t3) -- (v3) ;
		\draw [green] (s3) -- (v3) ;
		\draw [green] (t3) -- (u3) ;
		
		\node [label=left: $s_4$] (s4) at (2, 1){};
		\node [label=right: $t_4$] (t4) at (3, 1){};
		\node [label=left: $u_4$] (u4) at (2, 0){};
		\node [label=right: $v_4$] (v4) at (3, 0){};
		\draw [red] (s4) -- (t4) ;
		\draw [red] (u4) -- (v4) ;
		\draw [blue] (s4) -- (u4) ;
		\draw [blue] (t4) -- (v4) ;
		\draw [green] (s4) -- (v4) ;
		\draw [green] (t4) -- (u4) ;
		
		\node [label=left: $s_1^T$] (s1T) at (5,9){};
		\node [label=right: $t_1^T$] (t1T) at (6,9){};
		\node [label=left: $u_1^T$] (u1T) at (5,8){};
		\node [label=right: $v_1^T$] (v1T) at (6,8){};
		\draw [red] (s1T) -- (t1T) ;
		\draw [red] (u1T) -- (v1T) ;
		\draw [blue] (s1T) -- (u1T) ;
		\draw [blue] (t1T) -- (v1T) ;
		\draw [green] (s1T) -- (v1T) ;
		\draw [green] (t1T) -- (u1T) ;
		
		\node [label=left: $s_2^T$] (s2T) at (9,9){};
		\node [label=right: $t_2^T$] (t2T) at (10,9){};
		\node [label=left: $u_2^T$] (u2T) at (9,8){};
		\node [label=right: $v_2^T$] (v2T) at (10,8){};
		\draw [red] (s2T) -- (t2T) ;
		\draw [red] (u2T) -- (v2T) ;
		\draw [blue] (s2T) -- (u2T) ;
		\draw [blue] (t2T) -- (v2T) ;
		\draw [green] (s2T) -- (v2T) ;
		\draw [green] (t2T) -- (u2T) ;
		
		\node [label=left: $s_3^T$] (s3T) at (5,-1){};
		\node [label=right: $t_3^T$] (t3T) at (6,-1){};
		\node [label=left: $u_3^T$] (u3T) at (5,-2){};
		\node [label=right: $v_3^T$] (v3T) at (6,-2){};
		\draw [red] (s3T) -- (t3T) ;
		\draw [red] (u3T) -- (v3T) ;
		\draw [blue] (s3T) -- (u3T);
		\draw [blue] (t3T) -- (v3T);
		\draw [green] (s3T) -- (v3T) ;
		\draw [green] (t3T) -- (u3T) ;
		
		\node [label=left: $s_4^T$] (s4T) at (9,-1){};
		\node [label=right: $t_4^T$] (t4T) at (10,-1){};
		\node [label=left: $u_4^T$] (u4T) at (9,-2){};
		\node [label=right: $v_4^T$] (v4T) at (10,-2){};
		\draw [red] (s4T) -- (t4T) ;
		\draw [red] (u4T) -- (v4T) ;
		\draw [blue] (s4T) -- (u4T);
		\draw [blue] (t4T) -- (v4T);
		\draw [green] (s4T) -- (v4T) ;
		\draw [green] (t4T) -- (u4T) ;
		
		\node [label=left: $r_1^T$] (r1T) at (7,4){};
		\node [label=right: $r_2^T$] (r2T) at (8,4){};
		\node [label=left: $r_3^T$] (r3T) at (7,3){};
		\node [label=right: $r_4^T$] (r4T) at (8,3){};
		\draw [red] (r1T) -- (r2T) ;
		\draw [red] (r3T) -- (r4T) ;
		\draw [blue] (r1T) -- (r3T) ;
		\draw [blue] (r2T) -- (r4T) ;
		\draw [green] (r1T) -- (r4T) ;
		\draw [green] (r2T) -- (r3T) ;
		
		\draw (r1T) -- (s1T) node[draw=none,fill=none, midway, left] {$x$} ;
		\draw (r1T) edge[bend right =30]  node[draw=none,fill=none, midway,left] {$y$} (t1T) ;
		\draw (r1T) edge[bend left =15]  node[draw=none,fill=none, midway, left] {$z$} (u1T);
		\draw (r1T) -- (v1T) node[draw=none,fill=none, midway, right] {$w$} ;
		
		\draw (r2T) edge[bend left =25]  node[draw=none,fill=none, midway, right] {$x$} (s2T);
		\draw (r2T) -- (t2T) node[draw=none,fill=none, midway,right] {$y$} ;
		\draw (r2T) -- (u2T) node[draw=none,fill=none, midway, left] {$z$} ;
		\draw (r2T) edge[bend right =15]  node[draw=none,fill=none, midway, right] {$w$} (v2T);
		
		\draw (r3T) edge[bend right =15]  node[draw=none,fill=none, midway, left] {$x$} (s3T);
		\draw (r3T) -- (t3T) node[draw=none,fill=none, midway, right] {$y$} ;
		\draw (r3T) -- (u3T) node[draw=none,fill=none, midway, left] {$z$} ;
		\draw (r3T) edge[bend left =20]   node[draw=none,fill=none, midway, left] {$w$} (v3T);
		
		\draw (r4T) -- (s4T) node[draw=none,fill=none, midway, left] {$x$} ;
		\draw (r4T) edge[bend left =15]   node[draw=none,fill=none, midway,right] {$y$} (t4T);
		\draw (r4T) edge[bend right =25]  node[draw=none,fill=none, midway, right] {$z$} (u4T);
		\draw (r4T) -- (v4T) node[draw=none,fill=none, midway, right] {$w$} ;
	\end{tikzpicture}
		\caption{The construction from Lemma \ref{Gadget1}. The colored edge means the label of this edge is the vertex of the corresponding color. The right-hand side shows $H_0^T, \dots, H_4^T$ for $T = (x, y, z, w)$. }\label{Picture-Gadget1}
	\end{figure}

	\begin{lemma} \label{Gadget2}
			There is a linear 3-uniform hypergraph $G_2$ with specified vertices $a, b, c$ with the following properties:
		\begin{itemize}
			\item For every 3-coloring $f$ of $G_2$, we have $f(a), f(b), f(c)$ all distinct.
			\item $G_2$ is 3-colorable.
			\item There is a set $Z \subseteq V(G_2)$ with $|Z|\leq 19 $ such that $G_2 \setminus Z$ has no edges, and $a,b,c\in Z$.
			\item At most one edge of $G_2$ contains more than one of the vertices $a,b,c$.
		\end{itemize}
	\end{lemma}
	\begin{proof}
		Let $G_2$ be obtained from $G_1$ defined in Lemma \ref{Gadget1} by adding the edge $\{a,b,c\}$. The result follows immediately from Lemma \ref{Gadget1}.
	\end{proof}
	
	\begin{figure}[!hptb]
		\centering 
		\begin{tikzpicture}[scale=1]
			\node [label=left: $x$] (x) at (-5, 3.75){};
			\node [label=right: $y$] (y) at (8, 3.75){};
			
			\node [label=above: $s_1^{xy}$] (s1) at (0,1.5){};
			\node [label=above: $t_1^{xy}$] (t1) at (3,1.5){};
			\node [label=below: $u_1^{xy}$] (u1) at (0,0){};
			\node [label=below: $v_1^{xy}$] (v1) at (3,0){};
			\draw (s1) -- (t1) ;
			\node [draw=none, fill=none] () at (1.5,1.8) {$a_{2k-1}^2$};
			\draw (u1) -- (v1) ;
			\node [draw=none, fill=none] () at (1.5,-0.3) {$a_{2k}^2$};
			\draw (s1) -- node[draw=none,fill=none, midway,left] {$a_{2k-1}^3$} (u1) ;
			\draw (t1) -- node[draw=none,fill=none, midway,right] {$a_{2k-1}^3$} (v1) ;
			\draw (s1) -- (v1) ;
			\node [draw=none, fill=none] () at (0.5,0.9) {$a_{2k}^3$};
			\draw (t1) -- (u1) ;
			\node [draw=none, fill=none] () at (2.5,0.9) {$a_{2k}^3$};
			\draw (x) -- (s1) ;
			\node [draw=none, fill=none] () at (-2.2,2.2) {$a_{2k-1}^1$};
			\draw (x) edge [bend right = 10] (u1) ;
			\node [draw=none, fill=none] () at (-2.2,0.9) {$a_{2k}^1$};
			\draw (y) -- (t1) ;
			\node [draw=none, fill=none] () at (5.3,2.1) {$a_{2k-1}^1$};
			\draw (y) edge [bend left = 10] (v1) ;
			\node [draw=none, fill=none] () at (5.3,1.0) {$a_{2k}^1$};
			
			\node [label=above: $s_2^{xy}$] (s2) at (0,4.5){};
			\node [label=above: $t_2^{xy}$] (t2) at (3,4.5){};
			\node [label=below: $u_2^{xy}$] (u2) at (0,3){};
			\node [label=below: $v_2^{xy}$] (v2) at (3,3){};
			\draw (s2) -- (t2) ;
			\node [draw=none, fill=none] () at (1.5,4.8) {$a_{2k-1}^3$};
			\draw (u2) -- (v2) ;
			\node [draw=none, fill=none] () at (1.5,2.7) {$a_{2k}^3$};
			\draw (s2) -- node[draw=none,fill=none, midway,left] {$a_{2k-1}^1$} (u2) ;
			\draw (t2) -- node[draw=none,fill=none, midway,right] {$a_{2k-1}^1$} (v2) ;
			\draw (s2) -- (v2) ;
			\node [draw=none, fill=none] () at (0.5,3.9) {$a_{2k}^1$};
			\draw (t2) -- (u2) ;
			\node [draw=none, fill=none] () at (2.5,3.9) {$a_{2k}^1$};
			\draw (x) -- (s2) ;
			\node [draw=none, fill=none] () at (-2.2, 4.5) {$a_{2k-1}^2$};
			\draw (x) -- (u2) ;
			\node [draw=none, fill=none] () at (-2.2,3) {$a_{2k}^2$};
			\draw (y) -- (t2) ;
			\node [draw=none, fill=none] () at (5.3,4.5) {$a_{2k-1}^2$};
			\draw (y) -- (v2) ;
			\node [draw=none, fill=none] () at (5.3,3) {$a_{2k}^2$};
			
			\node [label=above: $s_3^{xy}$] (s3) at (0,7.5){};
			\node [label=above: $t_3^{xy}$] (t3) at (3,7.5){};
			\node [label=below: $u_3^{xy}$] (u3) at (0,6){};
			\node [label=below: $v_3^{xy}$] (v3) at (3,6){};
			\draw (s3) -- (t3) ;
			\node [draw=none, fill=none] () at (1.5,7.8) {$a_{2k-1}^1$};
			\draw (u3) -- (v3) ;
			\node [draw=none, fill=none] () at (1.5, 5.7) {$a_{2k}^1$};
			\draw (s3) -- node[draw=none,fill=none, midway,left] {$a_{2k-1}^2$} (u3) ;
			\draw (t3) -- node[draw=none,fill=none, midway,right] {$a_{2k-1}^2$} (v3) ;
			\draw (s3) -- (v3) ;
			\node [draw=none, fill=none] () at (0.5,6.9) {$a_{2k}^2$};
			\draw (t3) -- (u3) ;
			\node [draw=none, fill=none] () at (2.5,6.9) {$a_{2k}^2$};
			\draw (x) edge [bend left = 10]  (s3) ;
			\node [draw=none, fill=none] () at (-2.2, 6.8) {$a_{2k-1}^3$};
			\draw (x) -- (u3) ;
			\node [draw=none, fill=none] () at (-2.2,5.5) {$a_{2k}^3$};
			\draw (y) edge [bend right = 10]  (t3) ;
			\node [draw=none, fill=none] () at (5.3,6.7) {$a_{2k-1}^3$};
			\draw (y) -- (v3) ;
			\node [draw=none, fill=none] () at (5.3,5.4) {$a_{2k}^3$};

		\end{tikzpicture}
		\caption{The construction of $H^{xy}_3$, $H^{xy}_2$, $H^{xy}_1$ (top to bottom) for an edge $xy$ with $f'(xy)=k$. }\label{Picture-Edge}
	\end{figure}
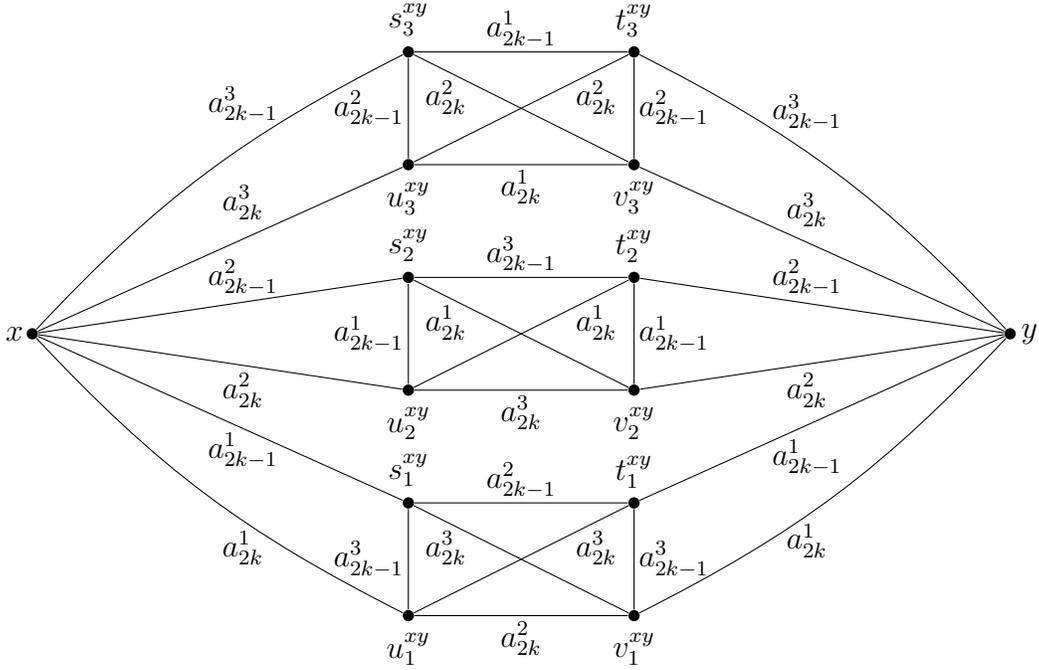
	Now we are ready to prove Theorem \ref{Linear-3ColNP}.
	\begin{proof}[Proof of Theorem \ref{Linear-3ColNP}]
		\setcounter{tbox}{0}
		
		We give an NP-hardness reduction from the \textsc{Graph 3-Coloring Problem} restricted to graphs with maximum degree at most 4, which is NP-hard by Theorem \ref{thm:3colbdd}. 

		Let $G^*$ be a graph with maximum degree at most $4$. Let $f':E(G^*)\rightarrow [5]$ be an edge-coloring of $G^*$. We construct a labeled graph representation $(G',l)$ of a 3-uniform linear hypergraph $G$ as follows.
		
		We create three sets of vertices $A=\{a_1^1,\dots,a_{10}^1\}$, $B=\{a_1^2,\dots,a_{10}^2\}$ and $C=\{a_1^3,\dots,a_{10}^3\}$. For the vertices $a_1^1,a_1^2,a_1^3$, we create a new copy of $G_2$ as defined in Lemma \ref{Gadget2}, denoted $G^{1}$, with $a_1^1,a_1^2,a_1^3$ as its specified vertices. For every $i \in \{2,\dots,10\}$, we create three new copies of $G_1$ as defined in Lemma \ref{Gadget1}, one with specified vertices $a_i^1,a_1^2,a_1^3$, one with specified vertices $a_1^1,a_i^2,a_1^3$ and one with specified vertices $a_1^1,a_1^2,a_i^3$, respectively. We denote these three hypergraphs $G^{i,1}$, $G^{i,2}$ and $G^{i,3}$ respectively. For convenience, we also define $G^{1,1}= G^{1,2}= G^{1,3}=G^{1}$.
		
		Next, for all $k \in [5]$ and for each edge $e=xy\in E(G^*)$ with $f'(xy)=k$, we create three copies of $K_4$, say $H_1^e, H_2^e, H_3^e$; see Figure \ref{Picture-Edge} for a picture of the construction described below. Let $V(H_i^e)=\{s_i^e, t_i^e, u_i^e, v_i^e\}$ for $i\in [3]$. Let $l(s_i^e t_i^e)=a_{2k-1}^{(i+1)}$, $l(s_i^e u_i^e)=l(t_i^e v_i^e)=a_{2k-1} ^{(i+2)}$, $l(u_i^e v_i^e)=a_{2k}^{(i+1)}$ and $l(s_i^e v_i^e)=l(t_i^e u_i^e)=a_{2k}^{(i+2)}$ for all $i \in [3]$, where superscripts are read modulo 3, so $a^{4}_{j} = a^{1}_{j}$ and $a^{5}_{j} = a^{2}_{j}$ for all $j\in[10]$. We also add edges $xs_i^e$, $yt_i^e$ with $l(xs_i^e)= l(yt_i^e)=a_{2k-1}^i$, and edges $xu_i^e$, $yv_i^e$ with $l(xu_i^e)= l(yv_i^e)=a_{2k}^i$ for all $i \in [3]$.
		
		Let $\mathcal{G} = \{G^1\} \cup  \{G^{i,j}: i\in \{2,\dots,10\}, j\in [3] \}$. Let $$U=(\cup_{G''\in \mathcal{G}} V(G''))\setminus (A\cup B\cup C), W=\cup_{e\in E(G^*)}\cup_{i\in [3]} V(H_i^e).$$ Let $V(G')=A\cup B\cup C \cup U\cup W\cup V(G^*)$ and let $E(G')$ be the set of all labeled edges defined above. By the construction, the function $l$ defined above is a labeling of $G'$. Let $G$ be the corresponding 3-uniform hypergraph of $(G',l)$. 
		
		Notice that from the construction, there is no other edge $e\in E(G)$ with $e\cap U\neq \emptyset$ and $e\cap (W \cup V(G^*))\neq \emptyset$. Furthermore, except for the edge $\{a_1^1, a_1^2, a_1^3\}$, there is no edge $e \in E(G)$ with $e\subseteq A\cup B\cup C\cup V(G^*)$. Moreover, for every edge $e \in E(G) \setminus \{a_1^1, a_1^2, a_1^3\}$, we have $|e \cap (A \cup B \cup C)| \leq 1$. Thus, for each edge $e\in E(G)\setminus \{a_1^1, a_1^2, a_1^3\}$, exactly one of the conditions $|e\cap U| \geq 2$ and $|e\cap (W \cup V(G^*))| = 2$ holds. Moreover, for all $e \in E(G)$, we have that $|e\cap V(G^*)|\leq 1$.
		
		\tbox{\label{eq:lin2}The 3-uniform hypergraph $G$ is linear.}
		
			We take two edges $e, e'\in E(G)$ with $e\neq e'$. Assume for a contradiction that $|e\cap e'|=2$. It follows that $e, e' \neq \{a_1^1, a_1^2, a_1^3\}$, since no edge except $\{a_1^1, a_1^2, a_1^3\}$ contains more than one vertex of $A \cup B \cup C$. 
			
			If $|e\cap U|\geq 2$, then $e \subseteq G^{a,b}$ for some $a\in[10]$ and $b\in[3]$. Since $|e\cap e'|=2$, we have that $e' \cap U \neq \emptyset$, and so $|e' \cap U| \geq 2$. It follows that $e'\subseteq V(G^{c,d})$ for some $c\in [\{2,\dots,10\}]$ and $d\in [3]$. By Lemma \ref{Gadget1}, $(a,b)\neq (c,d)$.  But then $V(G^{a,b}) \cap V(G^{c,d})\subseteq \{a_1^1,a_1^2,a_1^3\}$ and so $e \cap e' \subseteq \{a_1^1,a_1^2,a_1^3\}$. But $|e \cap \{a_1^1,a_1^2,a_1^3\}| \leq 1$, so $|e \cap e'| \leq 1$, which is a contradiction.

			If $|e\cap (W \cup V(G^*))| = 2$, then $|e\cap (A\cup B\cup C)|=1$. Since $|e\cap e'|=2$ and exactly one of $e'\cap U\neq \emptyset$ and $e'\cap (W \cup V(G^*)) \neq \emptyset$ holds, we have $e' \cap (W \cup V(G^*))\neq \emptyset$. It follows that $|e' \cap (W \cup V(G^*))| = 2$. Consider the labeled graph $G'$. Notice that by the construction above, for each $e^*\in E(G^*)$, no two edges of $G'[e^*\cup (\cup _{i\in [3]} V(H_i^{e^*}))]$ with the same label are incident to one common vertex. Thus $e$ and $e'$ are both incident to a common vertex $x\in V(G^*)$. For every $xy_1,xy_2\in E(G^*)$, since $f'$ is an edge coloring of $G^*$, $f'(xy_1)\neq f'(xy_2)$. Thus, for every two edges $e_1,e_2$ of $G'$ incident to $x$, $|e_1\cap e_2|=1$. Hence, we have proved $|e\cap e'|\leq1$, which leads to a contradiction. This proves \eqref{eq:lin2}. 

		\tbox{\label{eq:532}We have $\nu(G)\leq 532$.}

			By Lemmas \ref{Gadget1} and \ref{Gadget2}, for every graph $G'' = G^{i,j}\in \mathcal{G}$, there is a set $S_{G''}$ of size at most 19 which contains $a_i^j$ such that $G'' \setminus S_{G''}$ has no edges; for $G^1$, the set $S_{G^1}$ contains all of $a_1^1, a_1^2, a_1^3$. Each edge which is not a subset of $A\cup B\cup C \cup U$ contains a vertex in $A\cup B\cup C$. Thus, the set $X = \cup_{G''\in \mathcal{G}} S_{G''}$ meets all edges of $G$. Since $\mathcal{G}$ is a set of 28 graphs, it follows that $|X| \leq 19\cdot 28$. So $\nu(G)\leq 19\cdot 28 = 532$. This proves \eqref{eq:532}.

		\tbox{\label{eq:3col} The graph $G^*$ is 3-colorable if and only if $G$ is 3-colorable.}
		
			Let $c'$ be a 3-coloring of $G$. By Lemma \ref{Gadget2}, $c'(a_1^1)$, $c'(a_1^2)$ and $c'(a_1^3)$ are all distinct. Without loss of generality let $c'(a_1^1)=1$, $c'(a_1^2)=2$ and $c'(a_1^3)=3$. From the construction, by Lemma \ref{Gadget1}, $c'(a_i^1)=1$, $c'(a_i^2)=2$ and $c'(a_i^3)=3$ for all $i\in [10]$. We want to prove that $c'|_{V(G^*)}$ is a 3-coloring of $G^*$.
			
			Suppose for a contradiction that there exists an edge $xy\in E(G^*)$ with $c'(x)=c'(y)$. Let $k = f'(xy)$. Without loss of generality, let $c'(x)=c'(y)=1$. Then consider the graph $H_1^{xy}$. Because of the edges $\{x,s_1^{xy}, a_{2k-1}^1\}$, $\{x,u_1^{xy}, a_{2k}^1\}$, $\{y,t_1^{xy}, a_{2k-1}^1\}$ and 
			$\{y,v_1^{xy}, a_{2k}^1\}$, all of the vertices $s_1^{xy}, t_1^{xy}, u_1^{xy}, v_1^{xy}$ are colored 2 or 3. Since $c'(a_{2k-1}^3)=3$, from the edge $\{s_1^{xy}, u_1^{xy}, a_{2k-1}^3\}$, it follows that one of the vertices $s_1^{xy}, u_1^{xy}$ is not colored 3. Without loss of generality let $c'(s_1^{xy})=2$. Because of the edge $\{s_1^{xy}, t_1^{xy}, a_{2k-1}^2\}$, we have $c'(t_1^{xy})=3$. Consider the edges $\{t_1^{xy}, u_1^{xy}, a_{2k}^3\}$ and $\{t_1^{xy}, v_1^{xy}, a_{2k-1}^3\}$. Since $c'(a_{2k}^3)= c'(a_{2k-1}^3)= 3$, we have $c'(u_1^{xy})= c'(v_1^{xy})= 2$. But then the edge $\{u_1^{xy}, v_1^{xy}, a_{2k}^2\} $ is monochromatic, which contradicts the fact that $c'$ is a 3-coloring of $G$. This proves that if $G$ is $3$-colorable, then so is $G^*$.
			
			For the converse direction, let $c$ be a 3-coloring of $G^*$. We want to define a 3-coloring $d$ of $G$. 
			Let $d(v)=1$ for all $v\in A$, $d(v)=2$ for all $v\in B$, and $d(v)=3$ for all $v\in C$. By Lemmas \ref{Gadget1} and \ref{Gadget2}, there is a way to extend $d$ to $G[A \cup B \cup C \cup U]$.
			
			Let $d(v)=c(v)$ for all $v\in V(G^*)$. For each edge $xy\in E(G^*)$ and each $i\in [3]$, since $c$ is a 3-coloring of $G^*$, one of the vertices $x,y$ is not colored $i$. If $c(x)\neq i$, then for the set $V(H_i^{xy})$, we set $d(s_i^{xy})=  d(u_i^{xy})=i$ and $d(t_i^{xy})=  d(v_i^{xy})=i+1$, reading colors modulo 3 (so if this would assign color 4, we assign color 1 instead). If $c(x) = i$, then $c(y)\neq i$, and for the set $V(H_i^{xy})$, we set $d(s_i^{xy})=  d(u_i^{xy})=i+1$, again reading colors modulo 3; and $d(t_i^{xy})=  d(v_i^{xy})=i$. Thus, we have defined the function $d$ for all vertices of $G$. 
			
			We then want to show that $d$ is a 3-coloring of $G$. From the construction, all edges $e$ with $e\cap U \neq \emptyset$ are contained in $G[A \cup B \cup C \cup U]$ and hence not monochromatic. It remains to consider edges $e\in E(G)$ with $e\cap W \neq \emptyset$. It follows that there is an edge $xy \in E(G^*)$ and $i \in [3]$ such that $\emptyset \neq e \cap V(H_i^{xy}) = e \cap W$. If $x\in e$, then either $s_i^{xy} \in e$ or $t_i^{xy} \in e$ and from the construction of $d$, we have that $d(e \cap (A \cup B \cup C)) = \{i\}$, and either $d(x) \neq i$ or $d(s_i^{xy}), d(t_i^{xy}) \neq i$. The case $y \in e$ follows analogously. Therefore, we may assume that $|e \cap V(H_i^{xy})| = 2$. Now either the two vertices in $e \cap V(H_i^{xy})$ receive different colors, or they receive the same color in $\{i, i+1\}$ and $d(e \cap (A \cup B \cup C)) = \{i+2\}$. Thus, the edge $e$ is not monochromatic. This proves \eqref{eq:3col}. 
			
		\tbox{The 3-hypergraph $G$ can be constructed from $G^*$ in time $O(n^3)$, where $n=|V(G^*)|$.}
		
			Since $|V(G_1)|=O(1)$ and $|E(G_1)|=O(1)$, the 3-uniform hypergraph $G_1$ can be constructed in time $O(1)$. Similarly, the 3-uniform hypergraph $G_2$ can be constructed in time $O(1)$. We create $3\cdot(10)-2=28$ copies of the gadgets $G_1$ or $G_2$. This step can be done in time $O(1)$.
			
			Let $n=|V(G^*)|$, and $m=|E(G^*)|$. The edge coloring $f'$ of $G^*$ can be computed in time $O(mn)\leq O(n^3)$ by Theorem \ref{thm:vizing}. For each edge $e\in E(G^*)$, we create 12 new vertices and 30 edges. Thus, constructing the vertex set $W$ and all edges incident to $W$ takes time $O(n^2)$.
			
	\end{proof}

	\section*{Acknowledgments}
	We acknowledge the support of the Natural Sciences and Engineering Research Council of Canada (NSERC), [funding
reference number RGPIN-2020-03912]. \newline Cette recherche a \'et\'e financ\'ee par le Conseil de recherches en sciences naturelles et en g\'enie du Canada (CRSNG),
[num\'ero de r\'ef\'erence RGPIN-2020-03912].

	\bibliographystyle{abbrv}
\bibliography{reference}
\end{document}